\documentclass[letterpaper, 11pt, reqno]{amsart}
\usepackage[top = 1in, bottom = 1in, left=1in, right=1in]{geometry}
\usepackage{setspace} 
\usepackage[utf8]{inputenc}
\usepackage[T1]{fontenc}
\usepackage{lineno} 

\usepackage[inline]{enumitem}
\setlist[enumerate,1]{label={\upshape (\roman*)}}

\usepackage{amssymb,amsmath,amsthm,graphicx,xcolor,mathtools,mathrsfs}
\usepackage{comment}
\usepackage{fvextra}
\DefineVerbatimEnvironment{SageCode}{Verbatim}{
  fontsize=\footnotesize,
  xleftmargin=1.5em,
  breaklines=true,
  breakanywhere=false,
  breakindent=2em,
  breaksymbolleft={},
  breaksymbolright={},
  tabsize=4
}

\usepackage{hyperref}
\hypersetup{colorlinks, linkcolor={red!50!black}, citecolor={green!50!black}, urlcolor={blue!50!black}}

\usepackage[backend=biber, natbib=true, style=numeric, isbn=false, eprint=false, url=false, doi=false, maxbibnames=9]{biblatex}
\usepackage{csquotes}
\renewbibmacro{in:}{}
\DeclareFieldFormat{pages}{#1}
\DeclareFieldFormat[article,misc]{volume}{\mkbibbold{#1}\nopunct}
\DeclareFieldFormat[article,misc]{number}{(#1)}

\DeclareFieldFormat{title}{\mkbibquote{#1}}
\DeclareFieldFormat[misc]{date}{(#1)}
\DeclareFieldFormat[misc]{note}{#1\nopunct}

\renewbibmacro*{doi+eprint+url}{
	\iftoggle{bbx:url}
	{\iffieldundef{doi}{\usebibmacro{url+urldate}}{}}{}
	\newunit\newblock
	\iftoggle{bbx:eprint}
	{\usebibmacro{eprint}}
	{}
	\newunit\newblock
	\iftoggle{bbx:doi}
	{\printfield{doi}}
	{}
}

\newbibmacro{string+doi}[1]{%
	\iffieldundef{doi}{\iffieldundef{url}{#1}{\href{\thefield{url}}{#1}}}{\href{http://dx.doi.org/\thefield{doi}}{#1}}}	
\DeclareFieldFormat*{title}{\usebibmacro{string+doi}{#1}}

\usepackage{aliascnt}
\usepackage[nameinlink]{cleveref}
\theoremstyle{plain}
\newtheorem{theorem}{Theorem}[section]
\crefname{theorem}{Theorem}{Theorems}

\newcommand{\newaliastheorem}[3]{%
	  \newaliascnt{#1}{theorem}%
	  \newtheorem{#1}[#1]{#2}%
	  \aliascntresetthe{#1}%
	  \crefname{#1}{#2}{#3}%
	}

\newaliastheorem{proposition}{Proposition}{Propositions}
\newaliastheorem{corollary}{Corollary}{Corollaries}
\newaliastheorem{lemma}{Lemma}{Lemmas}
\newaliastheorem{conjecture}{Conjecture}{Conjectures}
\newaliastheorem{problem}{Problem}{Problems}
\newaliastheorem{claim}{Claim}{Claims}
\newaliastheorem{observation}{Observation}{Observations}
\newaliastheorem{setup}{Setup}{Setups}
\newaliastheorem{myth}{Myth}{Myths}
\newaliastheorem{fact}{Fact}{Facts}
\newaliastheorem{algorithm}{Algorithm}{Algorithms}
\newaliastheorem{remark}{Remark}{Remarks}
\newaliastheorem{example}{Example}{Examples}

\theoremstyle{definition}
\newaliastheorem{definition}{Definition}{Definitions}
\newaliastheorem{construction}{Construction}{Constructions}
\newaliastheorem{question}{Question}{Questions}

\numberwithin{equation}{section}

\crefname{section}{Section}{Sections}
\newcommand\cp{{\mathcal{P}}}
\newcommand\cc{{\mathcal{C}}}
\newcommand\ca{{\mathcal{A}}}
\allowdisplaybreaks

\usepackage{tikz}
\newlength{\myscaledsize}

\newcommand{\vvv}{	\setlength{\myscaledsize}{\the\fontdimen6\font}
	\hspace*{-0.08em}\resizebox{0.5\myscaledsize}{0.4\myscaledsize}{
		\tikz{
			\draw[black,fill=black] (0,0) circle (.2);
			\draw[black,fill=black] (1,0) circle (.2);
			\draw[black,fill=black] (0.5,0.86) circle (.2);}}
}
\newcommand{\VVV}{	\setlength{\myscaledsize}{\the\fontdimen6\font}
	\hspace*{-0.0444em}\raisebox{-0.1\myscaledsize}{\resizebox{0.75\myscaledsize}{0.6\myscaledsize}{
		\tikz{
			\draw[black,fill=black] (0,0) circle (.2);
			\draw[black,fill=black] (1,0) circle (.2);
			\draw[black,fill=black] (0.5,0.86) circle (.2);}}}
			\hspace{0.1em}
}
\newcommand{\piv}{\pi_{\vvv}}

\newcommand{\ve}{	\setlength{\myscaledsize}{\the\fontdimen6\font}
	\hspace*{-0.08em}\resizebox{0.5\myscaledsize}{0.4\myscaledsize}{
		\tikz{
			\draw[black,fill=black] (0,0) circle (.2);
			\draw[black,fill=black] (1,0) circle (.2);
			\draw[black,fill=black] (0.5,0.86) circle (.2);
			\draw[black,line width=4pt ](0,0)--(1,0);}}
}
\newcommand{\VE}{	\setlength{\myscaledsize}{\the\fontdimen6\font}
	\hspace*{-0.0444em}\raisebox{-0.1\myscaledsize}{\resizebox{0.75\myscaledsize}{0.6\myscaledsize}{
			\tikz{
				\draw[black,fill=black] (0,0) circle (.2);
				\draw[black,fill=black] (1,0) circle (.2);
				\draw[black,fill=black] (0.5,0.86) circle (.2);
				\draw[black,line width=4pt ](0,0)--(1,0);}}}
	\hspace{0.1em}
}
\newcommand{\pive}{\pi_{\ve}}

\newcommand\pa{{\pi_{\vvv}^{\operatorname{pal}}}}
\newcommand{\pos}[1]{\left(#1\right)_{+}}

\author[H. Lin]{Hao Lin}
\address{Department of Information Security, Naval University of Engineering, Wuhan, China.}
\email{baronlin001@gmail.com}
\author[W. Zhou]{Wenling Zhou}
\thanks{W. Zhou was supported by the Natural Science Foundation of China  (12401457), the China Postdoctoral Science Foundation (2024M761780), the Natural Science Foundation of Shandong Province (ZR2024QA067) and Young Talent of Lifting engineering for Science and Technology in Shandong, China (SDAST2025QTA074).}
\address{School of Mathematics, Shandong University, Jinan,  China.}
\email{gracezhou@sdu.edu.cn}

\title{Turán density of stars in uniformly dense hypergraphs}

\begin{document}
%\pagewiselinenumbers

\begin{abstract}
A $3$-uniform hypergraph (or $3$-graph) $H=(V,E)$ is $(d,\mu,\VVV)$-\emph{dense} if for any subsets $X,Y,Z\subseteq V$, the number of triples $(x,y,z)\in X\times Y\times Z$ such that $\{x,y,z\}$ is an edge of $H$ is at least $d|X||Y||Z|-\mu |V|^3$.
The \emph{$k$-star} $S_k$ is the $3$-graph with a center vertex and $k$ distinct leaf vertices, whose edge set consists of all triples containing the center and two distinct leaves.
Restricting to $\VVV$-dense $3$-graphs, determining the \emph{$\VVV$-uniform Tur\'an density} $\piv(S_k)$ of $S_k$ for $k\ge 4$ was proposed by Schacht in ICM 2022. In particular,
Reiher, R\"odl and Schacht gave a palette construction showing that $\piv(S_k)\ge \frac{k^2-5k+7}{(k-1)^2}$ for $k\ge 3$, and also proved that $\piv(S_3)=1/4$. 
Lamaison and Wu later showed that this palette construction is optimal for $k\ge 48$.

In this paper, we improve the results of Lamaison and Wu by proving that
\[ \piv(S_k)=\frac{k^2-5k+7}{(k-1)^2} \qquad\text{for all } k\ge 9. \]
\end{abstract}

\maketitle
\thispagestyle{empty}

\vspace{-2.15em}

\section{Introduction}
Tur\'an problems, which ask for the minimum density threshold for the existence of a certain substructure, are one of the most fundamental kinds of problems in extremal combinatorics — and among these, Tur\'an problems of hypergraphs are notoriously difficult.
Most known and conjectured extremal constructions for Tur\'an problems of hypergraphs contain large independent sets, i.e., linear-size sets of vertices without edges.  
This led Erd\H{o}s and S\'os~\cite{E-Sos-82} in the 1980s to propose a variant of this problem, restricting attention to $F$-free $3$-graphs that are uniformly dense on large subsets of the vertices. 
Formally, for $d \in [0, 1]$ and $\mu>0$, we say that a $3$-graph $H=(V, E)$ is \emph{$(d,\mu,\VVV)$-dense} if for any subsets $X, Y, Z\subseteq V$, the number $e_{\vvv}(X, Y, Z)$ of triples $(x,y,z)\in X\times Y\times Z$ with $xyz\in E$ satisfies
\[ e_{\vvv}(X, Y, Z)\ge d|X||Y||Z|-\mu |V|^3. \] 
Restricting to $\VVV$-dense $3$-graphs, the \emph{$\VVV$-uniform Tur\'an density} $\piv(F)$ for a given $3$-graph $F$ is defined as
\[ \begin{split}
\piv(F) = \sup \{ d\in [0,1] &: \text{for\ every\ } \mu>0 \ \text{and\ } n_0\in \mathbb{N},\ \text{there\ exists\ an\ } F \text{-free,} \\
&\quad (d,\mu, \VVV)\text{-dense}~  \text{$3$-graph~} H\ \text{with~} |V(H)|\geq n_0 \}.
\end{split} \]
With this notation, Erd\H{o}s and S\'os asked to determine $\piv(K^{(3)-}_4)$ and $\piv(K^{(3)}_4)$, where $K^{(3)-}_4$ denotes the $3$-graph consisting of three edges on four vertices and $K^{(3)}_4$ denotes the complete $3$-graph on four vertices. 
However, determining $\piv(F)$ of a given $3$-graph $F$ is also very challenging. 
For instance, the conjecture that $\piv(K^{(3)}_4)=1/2$ has remained an open problem in this area since R\"odl~\cite{k43-rodl} proposed a quasi-random construction in 1986.
For $\piv(K^{(3)-}_4)=1/4$, it was solved by Glebov, Kr\'al' and Volec~\cite{k43minus-1} using the flag algebra method of Razborov~\cite{Razborov-07}, and independently by Reiher, R\"odl and Schacht~\cite{k43minus-2} via the hypergraph regularity method. 
The latter set of authors also suggested a natural extension of $K^{(3)-}_4$.  
For $k\ge 3$, a {\it $k$-star}, denoted by $S_k$, is a $3$-graph on $(k+1)$ vertices $u, v_1,\dots, v_k$ such that $uv_iv_j\in E(S_k)$ for all $1\le i < j\le k$. Clearly, $S_3=K^{(3)-}_4$. 
A generalization of the argument in~\cite{k43minus-2} leads to 
\begin{equation}\label{eq:stars}
\frac{k^2-5k+7}{(k-1)^2}\leq \piv (S_k)\leq \left(\frac{k-2}{k-1}\right)^2 
\end{equation}
for all $k\ge 3$. 
When $k=3$, the upper and lower bounds coincide; however, for $k \ge 4$, a substantial gap remains. 
Consequently, Schacht~\cite{Schacht-ICM} posed the following problem in his lecture at the 2022 International Congress of Mathematicians (ICM):
\begin{problem}[{\cite[Problem 3.4]{Schacht-ICM}}]
Determine $\piv(S_k)$ for $k\ge 4$. 
\end{problem}

For the above problem, the best known result was obtained by Lamaison and Wu~\cite{LW24}, who proved that the lower bound in~\eqref{eq:stars} is sharp for all $k\ge 48$. 
Using computer-assisted computations, they further remarked that the lower bound in~\eqref{eq:stars} appears to be sharp already for $k \ge 40$.
In this paper, we extend the approach from~\cite{LW24} and prove that the palette construction giving the lower bound in~\eqref{eq:stars} is already optimal for all $k\ge 9$.

\begin{theorem}\label{thm:main-1}
$\piv (S_k)=\frac{k^2-5k+7}{(k-1)^2}$ for all $k\ge 9$.
\end{theorem}

Let us also briefly mention the closely related $\VE$-uniform version of the problem for stars.
Reiher, R\"odl and Schacht~\cite{RRS-k43} considered the following stronger density notion. 
Given $d\in[0,1]$ and $\mu>0$, a $3$-graph $H=(V,E)$ is \emph{$(d,\mu,\VE)$-dense} if for every subset $X\subseteq V$ and every set $P\subseteq V\times V$ of ordered pairs, the number $e_{\ve}(X,P)$ of pairs $(x,(y,z))\in X\times P$ with $xyz\in E$ satisfies
\[ e_{\ve}(X,P)\ge d|X||P|-\mu |V|^3. \]
The corresponding uniform Tur\'an density is defined analogously and denoted by $\pive(F)$.
Schacht~\cite[Problem 4.3]{Schacht-ICM} asked the analogous question of determining $\pive(S_k)$ for $k\ge4$.
Since every $(d,\mu,\VE)$-dense $3$-graph is $(d,\mu,\VVV)$-dense by taking $P=Y\times Z$, we have
\[ \pive(F)\le \piv(F) \]
for every $3$-graph $F$.
On the other hand, the palette construction giving the lower bound in~\eqref{eq:stars} also works in the $\VE$-setting, and hence $\pive(S_k)\ge \frac{k^2-5k+7}{(k-1)^2}$ for all $k\ge3$.
Therefore, \cref{thm:main-1} immediately yields the dot-edge analogue:
\begin{corollary}
$\pive (S_k)=\frac{k^2-5k+7}{(k-1)^2}$ for all $k\ge 9$.
\end{corollary}

For further results on $\piv$-uniform Tur\'an densities of $3$-graphs, we refer to~\cite{Daniel-cycle23, Chen-Bjarne-22, Daniel-8-27-24, Daniel-mini-24, Ander24, Hypergraph-llwzhou23, RRS-vanishing},  and to the excellent survey by Reiher~\cite{reiher2020extremal} for a more comprehensive treatment of the topic.

Unlike most results in the area, the proof of~\cref{thm:main-1} does not rely on the hypergraph regularity lemma or on reduced hypergraphs.
Instead, we employ a recently introduced framework due to Lamaison~\cite{Ander24}, 
which is based on \emph{palette constructions} for determining the $\VVV$-uniform Tur\'an density. 
In the next section we collect the necessary notation and tools, recall the palette construction giving the lower bound in~\eqref{eq:stars}, and reduce~\cref{thm:main-1} to an upper bound for palettes. In~\cref{section:proof-main} we prove this palette upper bound. The argument splits into the range $k\ge 11$ and the exceptional cases $k\in\{9,10\}$; in both ranges the main input is a local analytic majorant whose proof is deferred to~\cref{section:majorant-proofs}. Finally,~\cref{section:concluding} concludes with remarks and a conjecture.

\section{Preliminaries} \label{section:preliminaries}
For a positive integer $\ell$, we denote by $[\ell]$ the set $\{1,\dots,\ell\}$.
The key concept in our proof is the notion of palette, introduced by Reiher~\cite{reiher2020extremal} as a generalization of a construction of R\"odl~\cite{k43-rodl}.
\begin{definition}\label{def:palette}
A \emph{palette} $\mathcal P$ is a pair $(\mathcal C,\mathcal A)$, where $\mathcal C$ is a finite color set and $\mathcal A\subseteq \mathcal C\times \mathcal C\times \mathcal C$ is a set of ordered triples of colors. The \emph{density} of $\mathcal P$ is
\[ d(\mathcal P):={|\mathcal A|}/{|\mathcal C|^3}. \]
Moreover, given a color $a\in\mathcal C$ and $i\in[3]$, let $\mathcal A_a^i:=\{(c_1,c_2,c_3)\in\mathcal A: c_i=a\}$. The \emph{minimum degree} of $\mathcal P$ is
\[ \delta(\mathcal P):=\min_{i\in[3],\,a\in\mathcal C} {|\mathcal A_a^i|}/{|\mathcal C|^2}. \]
	
Given a palette $\mathcal P=(\mathcal C,\mathcal A)$ and a $3$-graph $F$, we say that $\mathcal P$ is \emph{$F$-good}, if there exist an order $\preceq$ on $V(F)$ and a function $\varphi:\binom{V(F)}{2}\to\mathcal C$ such that for every edge $uvw\in E(F)$ with $u\prec v\prec w$, we have $(\varphi(uv),\varphi(uw),\varphi(vw))\in\mathcal A$. Otherwise, we say that $\mathcal P$ is \emph{$F$-bad}.
\end{definition}

In~\cite{Ander24}, Lamaison established a general result that reduces the problem of determining $\VVV$-uniform Tur\'an densities of $3$-graphs to the corresponding palette problem.
\begin{theorem}[{\cite[Theorem 1.1]{Ander24}}]\label{thm:ander-dot}
For every $3$-graph $F$, let
\[ \pa(F):=\sup\{d(\mathcal P):\mathcal P \text{ is an } F\text{-bad palette}\}. \]
Then $\piv(F)=\pa(F)$.
\end{theorem}

\subsection{The lower-bound construction}\label{subsec:lower-bound-construction}
The lower bound in~\eqref{eq:stars} follows from a palette construction of Reiher, R\"odl and Schacht~\cite{k43minus-2}; see also the short presentation in~\cite{LW24}. We recall the construction here for completeness.

Let $k\ge 3$ and consider the palette $\mathcal P_k=(\mathcal C_k,\mathcal A_k)$ with $\mathcal C_k=\mathbb Z_{k-1}=\{0,1,\dots,k-2\}$, where all congruences below are taken modulo $k-1$, and
\[ \mathcal A_k:=\{(x,y,z)\in\mathbb Z_{k-1}^3: x\neq y,\ y\neq z,\ z\neq x+1\}. \]
For each fixed $x\in\mathbb Z_{k-1}$, there is one choice $y=x+1$, which gives $k-2$ possible choices for $z$, and there are $k-3$ choices of $y$ with $y\neq x$ and $y\neq x+1$, each of which gives $k-3$ possible choices for $z$. 
Hence
\[ d(\mathcal P_k)= \frac{|\mathcal A_k|}{|\mathcal C_k|^3} =  \frac{(k-1)\bigl((k-2) + (k-3)^2\bigr)}{(k-1)^3} = \frac{k^2-5k+7}{(k-1)^2}. \]

We now check that $\mathcal P_k$ is $S_k$-bad. Suppose, for a contradiction, that $\mathcal P_k$ is $S_k$-good . Let $u$ be the centre of $S_k$, and let $v_1,\dots,v_k$ be its leaves. By relabelling the leaves, we may assume that the order that witnesses $S_k$-good is
\[ v_1\prec\cdots\prec v_t\prec u\prec v_{t+1}\prec\cdots\prec v_k \]
for some $0\le t\le k$. Consider the following $k$ elements of $\mathbb Z_{k-1}$:
\[ \varphi(uv_1)+1,\dots,\varphi(uv_t)+1,
\varphi(uv_{t+1}),\dots,\varphi(uv_k). \]
By the pigeonhole principle, two of them are equal. We distinguish three cases. First, if $\varphi(uv_i)+1=\varphi(uv_j)+1$ for some $i<j\le t$, then $(\varphi(v_iv_j),\varphi(v_iu),\varphi(v_ju)) \notin \mathcal A_k$.
Second, if $\varphi(uv_i)+1=\varphi(uv_j)$ for some $i\le t<j$, then $(\varphi(v_iu),\varphi(v_iv_j),\varphi(uv_j)) \notin \mathcal A_k$; 
Finally, if $\varphi(uv_i)=\varphi(uv_j)$ for some $t<i<j$, then $(\varphi(uv_i),\varphi(uv_j),\varphi(v_iv_j)) \notin \mathcal A_k$.

In each case we obtain a contradiction. Thus $\mathcal P_k$ is $S_k$-bad, and consequently, by~\cref{thm:ander-dot},
\begin{equation}\label{eq:palette-lower-bound}
\piv(S_k)=\pa(S_k)\ge d(\mathcal P_k)
\qquad\text{for all }k\ge 3.
\end{equation}

\subsection{Auxiliary digraphs and degree estimates}\label{subsec:auxiliary-digraphs}
We next recall the auxiliary digraph associated with a palette. This notion was introduced by Lamaison and Wu~\cite{LW24} in their study of $\pa(S_k)$.

Let $\mathcal P=(\cc,\ca)$ be a palette. Given colors $a,b\in\cc$ and distinct indices $i,j\in[3]$, we say that $(a,b)$ is \emph{$(i,j)$-admissible} if there exists a triple $(c_1,c_2,c_3)\in\ca$ such that $c_i=a$ and $c_j=b$. Let $d_{i,j}(a)$ denote the number of colors $b$ such that $(a,b)$ is $(i,j)$-admissible, and set
\[ e_{i,j}(a):={d_{i,j}(a)}/{|\cc|}. \]
For each $a\in\cc$, define
\begin{align*}
&m_A(a):=\max\{e_{2,3}(a),e_{3,2}(a)\}, &&m_B(a):=m_A(a)+e_{1,3}(a),\\
&m_C(a):=\max\{e_{1,2}(a),e_{2,1}(a)\}, &&m_D(a):=m_C(a)+e_{3,1}(a).
\end{align*}

For a palette $\mathcal P=(\cc,\ca)$, define its \emph{auxiliary digraph} $D_{\mathcal P}$ as follows.
Its vertex set is $\cc_1\cup\cc_2$, where $\cc_1$ and $\cc_2$ are two disjoint copies of $\cc$. For a color $a\in\cc$, we write $a^1\in\cc_1$ and $a^2\in\cc_2$ for its two copies. For every ordered pair $(a,b)\in\cc\times\cc$, we add arcs according to the following rules:
\begin{itemize}
\item add an arc $a^1\to b^1$ in $D_{\mathcal P}[\cc_1]$ if $(a,b)$ is $(2,3)$-admissible;
\item add an arc $a^2\to b^2$ in $D_{\mathcal P}[\cc_2]$ if $(a,b)$ is $(1,2)$-admissible;
\item add arcs $a^1\leftrightarrow b^2$ between $\cc_1$ and $\cc_2$ if $(a,b)$ is $(1,3)$-admissible.
\end{itemize}

Let $T_k$ denote the transitive tournament on $k$ vertices, i.e., the digraph on $[k]$ with arcs $ij$ for all $1\le i<j\le k$. We shall use the following facts from~\cite{LW24}.
\begin{lemma}\label{lem:LW-palette-digraph}
For any $k\ge 3$, if a palette $\mathcal P$ is $S_k$-bad, then $D_{\mathcal P}$ has no loop and $D_{\mathcal P}$ is $T_k$-free.
\end{lemma}

We also need the maximum number of arcs in a $T_k$-free digraph. The following lemma follows directly from a theorem of Brown and Harary~\cite{brown1970extremal}, who determined the Tur\'an number of tournaments.
\begin{lemma}\label{lem:turan number of tournaments}
For any $k\ge 3$ and any positive integer $n$, if $D$ is a loopless $n$-vertex $T_k$-free digraph, then the number of arcs of $D$ is at most
$\frac{k-2}{k-1}(n^2-\alpha^2)+2\binom{\alpha}{2}$, where $\alpha=n\pmod{k-1}$.
\end{lemma}

We shall also use the following Caro-Wei type estimate for transitive-tournament-free digraphs, which was proved in~\cite[Lemma~6]{LW24} by induction.
For a digraph $D$ on $n$ vertices and a vertex $v\in V(D)$, write $\mu_D(v):=\max\left\{\frac{d_D^+(v)}{n},\frac{d_D^-(v)}{n}\right\}$, where $d_D^+(v)$ and $d_D^-(v)$ denote the out-degree and in-degree of $v$ in $D$ respectively.
\begin{lemma}\label{lem:digraph-inverse}
If $D$ is a $T_k$-free digraph on $n$ vertices, then
\[\sum_{v\in V(D)}\frac{1}{1-\mu_D(v)}\le (k-1)n. \]
\end{lemma}

The definitions of $m_A,m_B,m_C,m_D$ and of the auxiliary digraph immediately imply the following useful consequence.

\begin{corollary}\label{cor:degree-sums}
Let $k\ge 3$ and $\mathcal P=(\cc,\ca)$ be an $S_k$-bad palette with $n=|\cc|$. Then
\[ \sum_{a\in\cc}\frac{1}{1-m_A(a)}\le (k-1)n, \qquad
	\sum_{a\in\cc}\frac{1}{1-m_C(a)}\le (k-1)n, \]
and
\[\sum_{a\in\cc}\left(\frac{1}{2-m_B(a)}+\frac{1}{2-m_D(a)}\right)\le (k-1)n.\]
\end{corollary}

\begin{proof}
By~\cref{lem:LW-palette-digraph}, the digraphs $D_{\mathcal P}$, $D_{\mathcal P}[\cc_1]$ and $D_{\mathcal P}[\cc_2]$ are all $T_k$-free. In $D_{\mathcal P}[\cc_1]$, the value of $\mu_{D_{\mathcal P}[\cc_1]}(a^1)$ is $m_A(a)$; in $D_{\mathcal P}[\cc_2]$, the value of $\mu_{D_{\mathcal P}[\cc_2]}(a^2)$ is $m_C(a)$. Applying~\cref{lem:digraph-inverse} to these two induced digraphs proves the first two inequalities.
	
For the full auxiliary digraph $D_{\mathcal P}$, which has $2n$ vertices, we have
\[ \mu_{D_{\mathcal P}}(a^1)=\frac{m_B(a)}{2}, \qquad
   \mu_{D_{\mathcal P}}(a^2)=\frac{m_D(a)}{2}. \]
Applying~\cref{lem:digraph-inverse} to $D_{\mathcal P}$ gives
\[ \sum_{a\in\cc}\left(\frac{1}{1-m_B(a)/2}+\frac{1}{1-m_D(a)/2}\right)
	\le 2(k-1)n, \]
which is equivalent to the final displayed inequality.
\end{proof}

We are now ready to state the palette upper bound that remains to be proved and use it to prove~\cref{thm:main-1}. The proof of this upper bound is postponed to the next section.
\begin{theorem}\label{thm:palette}
Let $k\ge 9$ be an integer. For any $S_k$-bad palette $\mathcal P$,
\[ d(\mathcal P)\le \frac{k^2-5k+7}{(k-1)^2}. \]
\end{theorem}

\begin{proof}[Proof of~\cref{thm:main-1}]
The lower bound $\piv(S_k)\ge \frac{k^2-5k+7}{(k-1)^2}$ for all $k\ge3$ follows from~\eqref{eq:palette-lower-bound}. For $k\ge9$, the upper bound follows from~\cref{thm:ander-dot,thm:palette}. Therefore
\[ \piv(S_k)=\frac{k^2-5k+7}{(k-1)^2} \]
for every $k\ge9$.
\end{proof}

\section{Proof of the main theorem}\label{section:proof-main}
In this section, we prove~\cref{thm:palette}. The proof has three steps. First we convert the density of a palette into an average of local quantities $f_1,f_2,f_3$. We then bound these local quantities by expressions that are compatible with the degree-sum estimates from~\cref{cor:degree-sums}; the form of the local bound is different for $k\ge11$ and for $k\in\{9,10\}$. Finally, a minimal-counterexample argument removes the additional minimum-degree assumptions.

Let $\mathcal P=(\cc,\ca)$ be a palette and set $n=|\cc|$. For each $a\in\cc$, define
\begin{align*}
&f_1(a):=e_{1,2}(a)e_{1,3}(a)-\frac12(e_{1,2}(a)+e_{1,3}(a)),\\
&f_2(a):=e_{2,1}(a)e_{2,3}(a)-\frac12(e_{2,1}(a)+e_{2,3}(a)),\\
&f_3(a):=e_{3,1}(a)e_{3,2}(a)-\frac12(e_{3,1}(a)+e_{3,2}(a)).
\end{align*}

The following density estimate also appears in Lamaison and Wu~\cite{LW24}. Since our subsequent estimates differ from theirs, we include the proof for completeness.

\begin{lemma}\label{lem:palette-density-ineq}
For every palette $\mathcal P=(\cc,\ca)$,
\[ d(\mathcal P)\le 1+\frac1n\sum_{a\in\cc}\bigl(f_1(a)+f_2(a)+f_3(a)\bigr). \]
\end{lemma}

\begin{proof}
Define
\begin{align*}
	X_1&=\{(a,b,c)\in\cc^3: \text{for all }d\in\cc,\ (d,b,c)\notin\ca\},\\
	X_2&=\{(a,b,c)\in\cc^3: \text{for all }d\in\cc,\ (a,d,c)\notin\ca\},\\
	X_3&=\{(a,b,c)\in\cc^3: \text{for all }d\in\cc,\ (a,b,d)\notin\ca\}.
\end{align*}
For any $(a,b,c)\in\ca$, we have $(a,b,c)\notin X_1\cup X_2\cup X_3$, and thus
\begin{align}\label{eq:upperbound-of-A-new}
	|\ca|&\le n^3-|X_1\cup X_2\cup X_3|\notag\\
	&\le n^3-|X_1|-|X_2|-|X_3|+|X_1\cap X_2|+|X_1\cap X_3|+|X_2\cap X_3|.
\end{align}
For $a\in\cc$ and $i\neq j\in[3]$, let $d'_{i,j}(a):=n-d_{i,j}(a)$. Note that $(a,b,c)\in X_1$ is equivalent to $(b,c)$ not being $(2,3)$-admissible. Hence
\[ |X_1|=n\sum_{a\in\cc}d'_{2,3}(a)=n\sum_{a\in\cc}d'_{3,2}(a)=\frac n2\left(\sum_{a\in\cc}d'_{2,3}(a)+\sum_{a\in\cc}d'_{3,2}(a)\right). \]
Moreover, $(a,b,c)\in X_1\cap X_2$ is equivalent to $(b,c)$ not being $(2,3)$-admissible and $(a,c)$ not being $(1,3)$-admissible, so
\[ |X_1\cap X_2|=\sum_{a\in\cc}d'_{3,1}(a)d'_{3,2}(a). \]
Using the same argument for $X_2$ and $X_3$, we obtain
\begin{align*}
	d(\mathcal P)
	&\le 1-\frac{1}{2n^2}\sum_{a\in\cc}\sum_{i\neq j}d'_{i,j}(a)
	+\frac{1}{n^3}\sum_{a\in\cc}\bigl(d'_{1,2}(a)d'_{1,3}(a)+d'_{2,1}(a)d'_{2,3}(a)+d'_{3,1}(a)d'_{3,2}(a)\bigr)\\
	&=1+\frac{1}{n^3}\sum_{a\in\cc}\left(d_{1,2}(a)d_{1,3}(a)+d_{2,1}(a)d_{2,3}(a)+d_{3,1}(a)d_{3,2}(a)-\frac n2\sum_{i\neq j}d_{i,j}(a)\right)\\
	&=1+\frac1n\sum_{a\in\cc}\bigl(f_1(a)+f_2(a)+f_3(a)\bigr),
\end{align*}
as required.
\end{proof}

\subsection{The case \texorpdfstring{$k\ge 11$}{k-ge-11}}\label{subsec:k-ge-11}
We first handle the range $k\ge11$. In this range the minimal-counterexample estimate at the end of the section will force minimum degree at least $1/4$. This threshold lets us replace each local error term by a simple positive-part quadratic.

For a real number $z$, let $\pos{z}=\max\{z,0\}$ denote its positive part.

\begin{lemma}\label{lem:positive-part-ineq}
For any $x,y\in[0,1]$ with $xy\ge \frac14$, one has
\[ xy-\frac{x+y}{2}
	\le \frac12\pos{x-\frac12}^{2}+\frac12\pos{y-\frac12}^{2}-\frac14. \]
\end{lemma}

\begin{proof}
Since $xy\ge\frac14$, it is impossible that both $x<\frac12$ and $y<\frac12$. We may therefore assume $x\ge\frac12$. If $y\ge\frac12$, then
\[ xy-\frac{x+y}{2}\le \frac{x^2+y^2}{2}-\frac{x+y}{2}
	=\frac12\left(x-\frac12\right)^2+\frac12\left(y-\frac12\right)^2-\frac14. \]
If $y\le\frac12$, then
	\[
	\frac12\left(x-\frac12\right)^2-\frac14-\left(xy-\frac{x+y}{2}\right)
	=\left(x-\frac12\right)\left(\frac x2-y+\frac14\right)\ge 0.
	\]
	This proves the lemma.
\end{proof}

\begin{lemma}\label{lem:Phi-density-bound}
	Let $\mathcal P=(\cc,\ca)$ be a palette with $\delta(\mathcal P)\ge\frac14$. Define
	\[
	\phi(x,y):=\pos{x-\frac12}^{2}+\frac12\pos{y-x-\frac12}^{2}.	\]
	Then
	\[
	d(\mathcal P)\le \frac14+\frac1n\sum_{a\in\cc}\bigl(\phi(m_A(a),m_B(a))+\phi(m_C(a),m_D(a))\bigr). \]
\end{lemma}

\begin{proof}
	By~\cref{lem:palette-density-ineq}, it is enough to estimate $f_1(a),f_2(a),f_3(a)$ for a fixed color $a\in\cc$. Since $|\ca_a^1|\le d_{1,2}(a)d_{1,3}(a)$, $|\ca_a^2|\le d_{2,1}(a)d_{2,3}(a)$ and $|\ca_a^3|\le d_{3,1}(a)d_{3,2}(a)$, the assumption $\delta(\mathcal P)\ge1/4$ implies that each of the products
	\[
	e_{1,2}(a)e_{1,3}(a),\qquad e_{2,1}(a)e_{2,3}(a),\qquad e_{3,1}(a)e_{3,2}(a)
	\]
	is at least $1/4$. Applying~\cref{lem:positive-part-ineq} to these three pairs gives
	\begin{align*}
		&f_1(a)\le \frac12\pos{e_{1,2}(a)-\frac12}^{2}+\frac12\pos{e_{1,3}(a)-\frac12}^{2}-\frac14,\\
		&f_2(a)\le \frac12\pos{e_{2,1}(a)-\frac12}^{2}+\frac12\pos{e_{2,3}(a)-\frac12}^{2}-\frac14,\\
		&f_3(a)\le \frac12\pos{e_{3,1}(a)-\frac12}^{2}+\frac12\pos{e_{3,2}(a)-\frac12}^{2}-\frac14.
	\end{align*}
	Using
	\[
	e_{2,3}(a),e_{3,2}(a)\le m_A(a),
	\qquad
	e_{1,2}(a),e_{2,1}(a)\le m_C(a),
	\]
	and
	\[
	e_{1,3}(a)=m_B(a)-m_A(a),
	\qquad
	e_{3,1}(a)=m_D(a)-m_C(a),
	\]
	we obtain
	\[
	f_1(a)+f_2(a)+f_3(a)
	\le \phi(m_A(a),m_B(a))+\phi(m_C(a),m_D(a))-\frac34.
	\]
	The result follows from~\cref{lem:palette-density-ineq} after summing over $a\in\cc$.
\end{proof}

\begin{lemma}\label{lem:Phi-majorant}
	Let $k\ge 11$ and define
	\[
	\alpha_k:=\frac{k-3}{2(k-1)^3},
	\qquad
	\beta_k:=\frac{2(k-3)}{(k-1)^3},
	\qquad
	\gamma_k:=\frac{3(k-3)(k-7)}{8(k-1)^2}.
	\]
	Then for all $0\le x\le 1$ and all $y$ with $x\le y\le x+1$,
	\[
	\phi(x,y)\le \frac{\alpha_k}{1-x}+\frac{\beta_k}{2-y}+\gamma_k.
	\]
\end{lemma}

The proof of~\cref{lem:Phi-majorant} is postponed to~\cref{section:majorant-proofs}. Assuming this analytic estimate, the palette bound for $k\ge11$ follows by summing the two reciprocal terms and applying~\cref{cor:degree-sums}.

\begin{theorem}\label{thm:palette-k-ge-11}
	Let $k\ge 11$. If $\mathcal P$ is $S_k$-bad and $\delta(\mathcal P)\ge\frac14$, then
	\[
	d(\mathcal P)\le \frac{k^2-5k+7}{(k-1)^2}.
	\]
\end{theorem}

\begin{proof}
	By~\cref{lem:Phi-density-bound,lem:Phi-majorant},
	\[
	d(\mathcal P) \le \frac14+\frac1n\sum_{a\in\cc}\left(\frac{\alpha_k}{1-m_A(a)}+\frac{\beta_k}{2-m_B(a)}+\gamma_k\right) +\frac1n\sum_{a\in\cc}\left(\frac{\alpha_k}{1-m_C(a)}+\frac{\beta_k}{2-m_D(a)}+\gamma_k\right).
	\]
	Using~\cref{cor:degree-sums}, we obtain
	\[
	d(\mathcal P) \le \frac14+2\alpha_k(k-1)+\beta_k(k-1)+2\gamma_k =\frac14+\frac{3(k-3)^2}{4(k-1)^2} =\frac{k^2-5k+7}{(k-1)^2}.
	\]
\end{proof}

\subsection{The case \texorpdfstring{$k = 9$}{k=9} and \texorpdfstring{$k = 10$}{k=10}}\label{subsec:k-10}
For $k=9,10$ the minimum degree available from the minimal-counterexample argument is below $1/4$, so the quadratic estimate used above is no longer strong enough. We use instead a sharper envelope that keeps track of which coordinate is below $1/2$.
\begin{lemma}\label{lem:k10-envelope}
Let $k\in\{9,10\}$, $\sigma_k:=\frac{k^2-10k+18}{(k-1)^2}$ and $\mathcal P=(\cc,\ca)$ be a palette with $\delta(\mathcal P)\ge \sigma_k$. Define
\[ \psi_k(x,y):=
	\begin{cases}
		xy-\dfrac{x+y}{2}, & y\ge\dfrac12,\\[4pt]
		\sigma_k-\dfrac y2-\dfrac{\sigma_k}{2y}, & y\le\dfrac12,
	\end{cases} \qquad \text{and} \qquad
	\chi_k(x,y):= \max\left\{xy-\frac{x+y}{2}, \sigma_k - \sqrt{\sigma_k}\right\}.
	\]
	Then
	\[d(\cp) \le 1 + \frac1n\sum_{a\in\cc}\bigl(\psi_k(m_A(a), e_{3,1}(a)) + \chi_k(m_A(a), m_C(a)) + \psi_k(m_C(a), e_{1, 3}(a)) \bigr).\]
\end{lemma}

\begin{proof}
	Since $|\ca_a^1|\le d_{1,2}(a)d_{1,3}(a)$, $|\ca_a^2|\le d_{2,1}(a)d_{2,3}(a)$ and $|\ca_a^3|\le d_{3,1}(a)d_{3,2}(a)$, the assumption $\delta(\mathcal P)\ge \sigma_k$ implies that each of the products
	\[
	e_{1,2}(a)e_{1,3}(a),\qquad e_{2,1}(a)e_{2,3}(a),\qquad e_{3,1}(a)e_{3,2}(a)
	\]
	is at least $\sigma_k$.
	For $f_1(a)=e_{1, 2}(a)e_{1, 3}(a)-(e_{1, 2}(a)+e_{1, 3}(a))/2$, the variable $e_{1, 2}(a)$ ranges over $[\frac{\sigma_k}{e_{1,3}(a)}, m_C(a)]$. Since this expression is affine in $e_{1, 2}(a)$ with slope $e_{1, 3}(a)-1/2$, its maximum is attained at $e_{1, 2}(a) = m_C(a)$ when $e_{1, 3}(a) \ge 1/2$, and at $e_{1, 2}(a) = \frac{\sigma_k}{e_{1,3}(a)}$ when $e_{1,3}(a)\le1/2$. Therefore $f_1(a) \le \psi_k(m_C(a), e_{1, 3}(a))$. The same argument gives $f_3(a) \le \psi_k(m_A(a), e_{3,1}(a))$.
	
	It remains to control $f_2(a) = e_{2, 1}(a)e_{2, 3}(a)-(e_{2, 1}(a)+e_{2, 3}(a))/2$ over $0 \le e_{2, 1}(a) \le m_C(a)$, $0\le e_{2, 3}(a) \le m_A(a)$, and $e_{2, 1}(a)e_{2, 3}(a) \ge \sigma_k$. If both $e_{2, 1}(a)$ and $e_{2, 3}(a)$ are at least $1/2$, then this expression is increasing in each variable, so
	\[
	f_2(a)\le m_A(a)m_C(a)-\frac{m_A(a)+m_C(a)}{2}.
	\]
	If instead $e_{2, 1}(a) \le1/2$, then for fixed $e_{2, 1}(a)$ the expression is decreasing in $e_{2, 3}(a)$, hence the maximum occurs at $e_{2, 3}(a) = \frac{\sigma_k}{e_{2,1}(a)}$ and equals
	\[ \sigma_k -\frac {e_{2,1}(a)}{2}- \frac{\sigma_k}{2e_{2,1}(a)}\le \sigma_k-\sqrt{\sigma_k}. \]
	The case $e_{2, 3}(a)\le1/2$ is symmetric. 
	Combining the three estimates proves the lemma.
\end{proof}

\begin{lemma}\label{lem:k10-majorant}
Let $k\in\{9,10\}$, set $\sigma_k:=\frac{k^2-10k+18}{(k-1)^2}$, and define \[
\alpha_k:=\frac{k-3}{2(k-1)^3},
\qquad
\beta_k:=\frac{2(k-3)}{(k-1)^3},
\qquad
\gamma_k:=\frac{3(2k - 5)}{(k-1)^2}.
\]
Then for all real numbers $x,y,s,t$ satisfying
\[ 0 \le x,y,s,t\le 1,
	\qquad
	xs\ge \sigma_k,\quad xy\ge \sigma_k,\quad yt\ge \sigma_k,
	\]
	one has
	\begin{align}\label{eq:k10-majorant}
		\psi_k(x,s)+\chi_k(x,y)+\psi_k(y,t)
		\le \alpha_k\left(\frac{1}{1-x}+\frac{1}{1-y}\right) +\beta_k\left(\frac{1}{2-y-s}+\frac{1}{2-x-t}\right)-\gamma_k.
	\end{align}
\end{lemma}

The proof of~\cref{lem:k10-majorant} is postponed to~\cref{section:majorant-proofs}. We now combine it with the envelope from~\cref{lem:k10-envelope}; this gives the desired palette bound in the two exceptional cases, provided the minimum degree is slightly larger than $\sigma_k$.

\begin{theorem}\label{thm:palette-k-10}
Let $k\in\{9,10\}$,	if $\mathcal P$ is $S_k$-bad and $\delta(\mathcal P)> \sigma_k$, then
	\[
	d(\mathcal P)\le \frac{k^2-5k+7}{(k-1)^2}.
	\]
\end{theorem}

\begin{proof}
	By~\cref{lem:palette-density-ineq,lem:k10-envelope,lem:k10-majorant}, we obtain
	\begin{align*}
		d(\mathcal P)
		&\le 1+\frac1n\sum_{a\in\cc}\bigl(f_1(a)+f_2(a)+f_3(a)\bigr)\\
		&\le 1+\frac1n\sum_{a\in\cc}\bigl(\psi_k(m_A(a), e_{3,1}(a)) + \chi_k(m_A(a), m_C(a)) + \psi_k(m_C(a), e_{1,3}(a))\bigr)\\
		&\le 1+\frac1n\sum_{a\in\cc}\left[
		\alpha_k\left(\frac{1}{1-m_C(a)}+\frac{1}{1-m_A(a)}\right)
		+\beta_k\left(\frac{1}{2-m_B(a)}+\frac{1}{2-m_D(a)}\right)-\gamma_k
		\right].
	\end{align*}
	Using~\cref{cor:degree-sums}, we conclude that
	\[
	d(\mathcal P)\le 1+ 2\alpha_k(k-1)+ \beta_k(k-1) -\gamma_k
	= \frac{k^2-5k+7}{(k-1)^2}.
	\]
\end{proof}

\subsection{Completion of the proof}\label{subsec:completion}
It remains to remove the minimum-degree assumptions in~\cref{thm:palette-k-ge-11,thm:palette-k-10}. The following standard minimal-counterexample estimate shows that any smallest counterexample must automatically have the required local density.

\begin{lemma}\label{lem:minimal-counterexample-degree}
	Let $k \ge 3$. Suppose that $\mathcal P=(\cc,\ca)$ is an $S_k$-bad palette with $d(\mathcal P)> \frac{k^2-5k+7}{(k-1)^2} $ and with $|\cc|$ minimum among all such palettes. Then
	\[ \delta(\mathcal P)\ge 3d(\mathcal P)-\frac{2k-3}{k-1}. \]
	Consequently, $\delta(\mathcal P)>1/4$ for $k\ge11$, and $\delta(\mathcal P)> \sigma_k$ for $k \in \{9, 10\}$.
\end{lemma}

\begin{proof}
	Let $|\cc|=n$ and $D_i:=D_{\mathcal P}[\cc_i]$ be the induced subdigraph of $D_{\mathcal P}$ on $\cc_i$ for $i\in[2]$. Combining~\cref{lem:LW-palette-digraph,lem:turan number of tournaments} gives
	\[|E(D_i)|\le \frac{k-2}{k-1}n^2 \]
	for $i\in[2]$. 
	We first show a simple lower bound for $d_{i,j}(a)$. Fix $a\in \cc$, delete the color $a$, and denote by $\mathcal P'=(\cc \setminus\{a\},\ca')$ the resulting palette. By the minimality of $|\cc|$,
	\[ |\ca'|\le (n-1)^3 d(\cp). \]
	By symmetry it is enough to treat $(i,j)=(1,2)$. Deleting $a$ removes all triples from $\ca_a^1$, all triples from $\ca_a^2\setminus \ca_a^1$, and all triples from $\ca_a^3\setminus (\ca_a^1\cup \ca_a^2)$. These three sets have sizes at most $nd_{1,2}(a)$, $n(n-1)$, and $(n-1)^2$, so
	\[n^3 d(\cp)-|\ca'|\le nd_{1,2}(a)+n(n-1)+(n-1)^2.\]
	Hence, for every ordered pair $(i,j)$ with $i\ne j$ one has
	\[nd_{i,j}(a) \ge (3d(\cp) - 2)n^2 + (3n-1)(1-d(\cp)) .\]
	
	Let $x \in\cc$ and $\ell\in[3]$ be such that $|\ca_x^\ell|=\delta(\mathcal P)n^2$. Let $\tilde{\cp} =(\tilde{\cc}, \tilde{\ca})$ be the palette obtained by removing the color $x$ from $\cc$. By symmetry, assume $\ell = 2$.
	Then
	\[ |\ca\setminus \tilde{\ca}|=|\ca_x^2|+|\ca_x^1\setminus \ca_x^2|+|\ca_x^3\setminus (\ca_x^1\cup \ca_x^2)|. \]
	According to the definition of the auxiliary digraph $D_{\mathcal P}$, we have
	\[|\ca_x^1\setminus \ca_x^2|\le |E(D_1)|-d_{2,3}(x), \qquad |\ca_x^3\setminus (\ca_x^1\cup \ca_x^2)|\le (n-1)^2.\]
	Hence
	\[ |\ca| \le |\tilde{\ca}|+\delta(\mathcal P)n^2 + |E(D_1)| - d_{2,3}(x) +(n-1)^2.\]
	Thus, we have
	\begin{align*}
		\delta(\mathcal P)n^2
		&\ge (n^3-(n-1)^3)d(\mathcal P)- \frac{k-2}{k-1}n^2 - (n - 1)^2 + (3d(\cp) - 2)n + \frac{(3n-1)(1-d(\cp))}{n} \\
		& = \left(3d(\mathcal P)-\frac{2k-3}{k-1}\right)n^2 + \frac{(2n-1)(1-d(\cp))}{n} \ge \left(3d(\mathcal P)-\frac{2k-3}{k-1}\right)n^2.
	\end{align*}
	
	This proves the displayed bound. If $d(\mathcal P)>\frac{k^2-5k+7}{(k-1)^2}$, then
	\[ \delta(\mathcal P)>3\cdot \frac{k^2-5k+7}{(k-1)^2}-\frac{2k-3}{k-1}
	=\frac{k^2-10k+18}{(k-1)^2}. \]
	The last expression is at least $1/4$ for $k\ge11$ and equals $\sigma_k$ for $k\in \{9, 10\}$.
\end{proof}

\begin{proof}[Proof of~\cref{thm:palette}]
	Assume for contradiction that there is an $S_k$-bad palette $\mathcal P$ with $d(\mathcal P)> \frac{k^2-5k+7}{(k-1)^2}$, where $k\ge 9$. Choose such a palette with $|\cc|$ minimum. By~\cref{lem:minimal-counterexample-degree}, if $k\ge11$, then $\delta(\mathcal P)>1/4$, contradicting~\cref{thm:palette-k-ge-11}. If $k\in \{9, 10\}$, then $\delta(\mathcal P)> \sigma_k$, contradicting~\cref{thm:palette-k-10}. Thus no such counterexample exists, and the theorem follows.
\end{proof}

\section{Proofs of the two analytic majorants}\label{section:majorant-proofs}
We now prove the two analytic estimates used in the preceding section.

\begin{proof}[Proof of~\cref{lem:Phi-majorant}]
	If $x=1$ or $y=2$, the right-hand side is interpreted as $+\infty$, so there is nothing to prove. Set
	\[
	u:=1-x\in(0,1],
	\qquad
	t:=y-x\in[0,1].
	\]
	The desired inequality is
	\[
	\Psi_k(u,t):=\frac{\alpha_k}{u}+\frac{\beta_k}{1+u-t}+\gamma_k
	-\pos{\frac12-u}^{2}-\frac12\pos{t-\frac12}^{2}\ge0.
	\]
	We split the square $0<u\le1$, $0\le t\le1$ into four regions.
	
	\smallskip
	\noindent\textbf{Region I: $u\ge\frac12$ and $t\le\frac12$.}
	Here both positive-part terms vanish, so $\Psi_k(u,t)>0$.
	
	\smallskip
	\noindent\textbf{Region II: $u\ge\frac12$ and $t\ge\frac12$.}
	For fixed $t$,
	\[
	\frac{\partial\Psi_k}{\partial u}(u,t)
	=-\frac{\alpha_k}{u^2}-\frac{\beta_k}{(1+u-t)^2}<0,
	\]
	so the minimum occurs at $u=1$. It is enough to check
	\[
	h_k(t):=\alpha_k+\frac{\beta_k}{2-t}+\gamma_k-\frac12\left(t-\frac12\right)^2\ge0
	\qquad\left(\frac12\le t\le1\right).
	\]
	Since $h_k''(t)=2\beta_k/(2-t)^3-1<0$ for every $k\ge11$, $h_k$ is concave. Hence it is enough to evaluate the endpoints:
	\[
	h_k\left(\frac12\right)=\frac{(k-3)(9k^2-72k+107)}{24(k-1)^3}>0, \qquad
	h_k(1)=\frac{k^3-15k^2+55k-61}{4(k-1)^3}>0.
	\]
	Thus Region~II is settled.
	
	\smallskip
	\noindent\textbf{Region III: $u\le\frac12$ and $t\le\frac12$.}
	For fixed $u$,
	\[
	\frac{\partial\Psi_k}{\partial t}(u,t)=\frac{\beta_k}{(1+u-t)^2}>0,
	\]
	so the minimum occurs at $t=0$. Hence it is enough to show
	\[
	\frac{\alpha_k}{u}+\frac{\beta_k}{1+u}+\gamma_k-\left(\frac12-u\right)^2\ge0
	\qquad\left(0<u\le\frac12\right).
	\]
	Set
	\[
	G(k,u):=8(k-1)^3u(1+u)\left(\frac{\alpha_k}{u}+\frac{\beta_k}{1+u}+\gamma_k-\left(\frac12-u\right)^2\right).
	\]
	Then
	\[
	G(k,u)=-8(k-1)^3u^4+(9k^3-51k^2+111k-69)u^2+(k^3-27k^2+107k-121)u+4(k-3).
	\]
	For fixed $u\in(0,1/2]$, one checks from the displayed formula that $\frac{\partial G}{\partial k}>0$ for all $k\ge11$. Hence it remains to show that
	\[
	G(11,u)=-16\bigl(500u^4-435u^2+55u-2\bigr)>0
	\]
	on $(0,1/2]$. Put $g(u):=500u^4-435u^2+55u-2$. We compute
	\[
	g'(u)=2000u^3-870u+55,
	\qquad
	g''(u)=6000u^2-870.
	\]
	Thus $g'$ is strictly decreasing on $(0,\sqrt{29/200})$ and strictly increasing on $(\sqrt{29/200},1/2)$. Since $g'(0)=55>0$ and $g'(1/2)=-130<0$, the only possible maximum of $g$ in $(0,1/2]$ occurs at a critical point $u_0$. At such a point,
	\[
	2000u_0^3=870u_0-55,
	\]
	and hence
	\[
	g(u_0)=-\frac{435}{2}u_0^2+\frac{165}{4}u_0-2
	\le -2+\frac{(165/4)^2}{4\cdot(435/2)}=-\frac{41}{928}<0.
	\]
	Therefore $g(u)<0$ throughout $(0,1/2]$, proving Region~III.
	
	\smallskip
	\noindent\textbf{Region IV: $u\le\frac12$ and $t\ge\frac12$.}
	Set
	\[
	\ell:=k-1\ge10,
	\qquad
	s:=\ell u,
	\qquad
	r:=\ell(1-t).
	\]
	Then $0<s\le\ell/2$ and $0\le r\le\ell/2$, and a direct simplification shows that
	\[
	\ell^2\Psi_k(u,t)=H(\ell,s,r),
	\]
	where
	\[
	H(\ell,s,r):=\frac{\ell-2}{2}\left(\frac{(s-1)^2}{s}+\frac{(s+r-2)^2}{s+r}\right)-(s-1)^2-\frac{(r-1)^2}{2}.
	\]
	For fixed $s,r$, the function $H(\ell,s,r)$ is affine and increasing in $\ell$, because
	\[
	\frac{\partial H}{\partial\ell}(\ell,s,r)=\frac12\left(\frac{(s-1)^2}{s}+\frac{(s+r-2)^2}{s+r}\right)\ge0.
	\]
	Since $\ell\ge\max\{10,2s,2r\}$, it suffices to prove the claim for $\ell_0:=\max\{10,2s,2r\}$. We split into three cases.
	
	\smallskip
	\emph{Case 1: $s,r\le5$.} Define
	\[
	h_s(r):=H(10,s,r)=4\left(\frac{(s-1)^2}{s}+\frac{(s+r-2)^2}{s+r}\right)-(s-1)^2-\frac{(r-1)^2}{2}.
	\]
	A direct computation gives
	\begin{align}\label{eq:region-4-boundary-validation}
	h_s(0)=\frac{-2s^3+20s^2-51s+40}{2s}\ge0, \qquad
	h_s(5)=\frac{-s^4+5s^3+37s^2-45s+20}{s(s+5)}\ge0,
	\end{align}
	for every $0<s\le5$. Moreover,
	\[
	h_s'(r)=\frac{(s+r)^2(5-r)-16}{(s+r)^2}, \qquad h_s''(r) = \frac{32 - (s+r)^3 }{(s+r)^3}.
	\]
	If $s\ge4/\sqrt5$, then $h_s'(0)\ge0$ and $h_s'(5)<0$, so $h_s$ first increases and then decreases, and has no interior minimum. Hence~\eqref{eq:region-4-boundary-validation} proves the claim.
	
	Now suppose $0<s<4/\sqrt5$. Then $h_s'(0)<0$ and $h_s'(3)>0$, so $h_s'$ has exactly two zeros on $[0,5]$. The first one, say $\rho\in(0,3)$, is a local minimum of $h_s$. Thus the minimum of $h_s$ is attained either at $r=5$ or at $r=\rho$. The endpoint has already been handled. At the critical point $\rho$, we have
	\[
	(s+\rho)^2(5-\rho)=16.
	\]
	Let $x=\sqrt{5-\rho}\in(\sqrt2,\sqrt5)$. Then $\rho=5-x^2$, $s+\rho=4/x$, and $s=x^2+4/x-5$. Substituting these into $h_s$ gives
	\[
	h_s(\rho)=-\frac{(x-2)^2P(x)}{2x^2(x-1)(x^2+x-4)},
	\]
	where
	\[
	P(x)=3x^7+12x^6-19x^5-104x^4+36x^3+200x^2-168x+32.
	\]
	A Sturm sequence computation gives $P(x)<0$ for all $x\in(\sqrt2,\sqrt5)$. Hence $h_s(\rho)>0$, and Case~1 is complete.
	
	\smallskip
	\emph{Case 2: $s\ge\max\{r,5\}$.} Let $F_s(r):=H(2s,s,r)$ on $0\le r\le s$. A short calculation gives
	\[
	F_s''(r)=-1+\frac{8(s-1)}{(s+r)^3}<0
	\qquad(s\ge5),
	\]
	so $F_s$ is concave and its minimum is attained at an endpoint. Now
	\[
	F_s(0)=s^2-6s+\frac{19}{2}-\frac5s,
	\qquad
	F_s(s)=\frac{3(s-2)(s-1)^2}{2s}.
	\]
	The second quantity is positive for $s\ge5$. The first is increasing for $s\ge5$, since
	\[
	\frac{dF_s(0)}{ds}=2s-6+\frac5{s^2}>0,
	\]
	so $F_s(0)\ge F_5(0)=7/2>0$. Hence $H(2s,s,r)>0$ in Case~2.
	
	\smallskip
	\emph{Case 3: $r\ge\max\{s,5\}$.} For fixed $s$, a direct computation gives:
	\[
	\frac{\partial H}{\partial r}(2r,s,r)=\frac{M(s,r)}{s(r+s)^2},
	\]
	where
\[ M(s,r):=r^3s+4r^2s^2-6r^2s+r^2+5rs^3-12rs^2+2rs +2s^4-6s^3+5s^2+4s. \]
	Moreover,
	\[
	\frac{\partial M}{\partial r}(s,r)=(r+s)(3rs+5s^2-12s+2)
	\ge(r+s)(5s^2+3s+2)>0
	\]
	whenever $r\ge5$. Thus $M(s,r)$ is increasing in $r$, and so
	\[
	M(s,r)\ge M(s,5)=2s^4+19s^3+45s^2-11s+25>0,
	\]
	because $45s^2-11s+25$ has negative discriminant. Therefore $\frac{\partial H}{\partial r}(2r,s,r)>0$ for $r\ge\max\{s,5\}$.
	Thus, if $s \le 5$, then
	\[
	H(2r,s,r)\ge H(10,s,5),
	\]
	which was already established in Case~1.
	If $s > 5$, then
	\[H(2r,s,r)\ge H(2s, s, s) = \frac{3(s - 1)^2(s-2)}{2s} > 0.\]
	Hence Case~3 follows.
	
	This completes the proof of Region~IV and of the lemma.
\end{proof}

\begin{proof}[Proof of~\cref{lem:k10-majorant}]
If one of the denominators on the right-hand side of~\eqref{eq:k10-majorant} is zero, then the right-hand side is interpreted as $+\infty$, and there is nothing to prove. 
Hence we may assume that all denominators are positive. 
Set
\[ a:=1-x,\qquad b:=1-y,\qquad p:=1-s,\qquad q:=1-t.\]
Then, we have 
\[ 0<a,b\le 1-\sigma_k,\quad 0 \le p, q \le 1-\sigma_k,\quad (1-a)(1-p)\ge \sigma_k, \quad (1-a)(1-b)\ge \sigma_k, \quad (1-b)(1-q)\ge \sigma_k. \]
Moreover,
\[ 2-y-s=b+p, \qquad 2-x-t=a+q. \]
	Define
	\[
	\Psi_k(a,p):=\psi_k(1-a,1-p)
	=
	\begin{cases}
		ap-\dfrac{a+p}{2}, & p\le\dfrac12,\\[4pt]
		\dfrac p2+\sigma_k-\dfrac12-\dfrac{\sigma_k}{2(1-p)}, & p\ge\dfrac12,
	\end{cases}
	\]
	and
	\[ \Phi_k(a,b):=\chi_k(1-a,1-b)
	=
	\max\left\{ab-\frac{a+b}{2},\ \sigma_k-\sqrt{\sigma_k} \right\}.
	\]
	Thus the desired inequality is equivalent to
	\begin{equation}\label{eq:k9k10-transformed}
		\Psi_k(a,p)+\Phi_k(a,b)+\Psi_k(b,q)
		\le
		\frac{\alpha_k}{a}+\frac{\alpha_k}{b}+\frac{\beta_k}{b+p}+\frac{\beta_k}{a+q}-\gamma_k.
	\end{equation}
	For convenience, put
	\[
	E_k(a,b,p):=\frac{\beta_k}{b+p}-\Psi_k(a,p).
	\]
	Then~\eqref{eq:k9k10-transformed} is equivalent to proving
	\[
	R_k:=
	\frac{\alpha_k}{a}+\frac{\alpha_k}{b}-\gamma_k-\Phi_k(a,b)
	+E_k(a,b,p)+E_k(b,a,q)
	\ge0.
	\]
	
Now, we first estimate the bounds for $E_k$. If $p\ge1/2$, then
	\begin{align}
		E_k(a,b,p)
		=\frac{\beta_k}{b+p}-\frac p2+\frac12-\sigma_k+\frac{\sigma_k}{2(1-p)}
		\ge
		\frac{\beta_k}{2(1-\sigma_k)}+ \sqrt{\sigma_k}-\sigma_k.
		\label{eq:k9k10-E-low}
	\end{align}
Here we used $b+p\le2(1-\sigma_k)$. 
If $p\le1/2$, then
	\begin{align}
		E_k(a,b,p)
		=\frac{\beta_k}{b+p}+a\left(\frac12-p\right)+\frac p2 
		\ge
		\frac{\beta_k}{\frac32-\sigma_k}+
		\min\left\{\frac a2,\frac14\right\}.
		\label{eq:k9k10-E-crude}
	\end{align}
Indeed, $b+p\le1-\sigma_k+1/2=3/2-\sigma_k$, and the second term is a convex combination of $a/2$ and $1/4$.

We now split according to which term realizes $\Phi_k(a,b)$.

\smallskip
\noindent\textbf{Case 1: $\Phi_k(a,b)= \sigma_k-\sqrt{\sigma_k}$.}
In this case
\[
R_k=
\frac{\alpha_k}{a}+\frac{\alpha_k}{b}-\gamma_k- \sigma_k + \sqrt{\sigma_k}
+E_k(a,b,p)+E_k(b,a,q).
\]
If $p,q\ge1/2$, then by~\eqref{eq:k9k10-E-low} and $a,b\le1-\sigma_k$ give
\[
R_k\ge
\frac{2\alpha_k + \beta_k}{1-\sigma_k}+ 3 \sqrt{\sigma_k} - 3\sigma_k-\gamma_k > 0,
\]
where the last inequality follows by directly substituting $k =9$ and $k =10$ into the calculation.

If exactly one of $p,q$ is at least $1/2$.  By symmetry, assume $p\ge1/2$ and $q\le1/2$.  Since $(1-a)(1-p)\ge \sigma_k$ and $p\ge1/2$, we have $a\le1-2\sigma_k$, and hence $\alpha_k/a\ge \alpha_k/(1-2\sigma_k)$. In this case, combining~\eqref{eq:k9k10-E-low} and~\eqref{eq:k9k10-E-crude} we have:
\[ R_k \ge \frac{\alpha_k}{1-2\sigma_k} - \gamma_k + 2\sqrt{\sigma_k} - 2\sigma_k + \frac{\beta_k}{2(1-\sigma_k)} + \frac{\beta_k}{\frac32-\sigma_k} +  \frac{\alpha_k}{b} + \min\left\{\frac{b}{2}, \frac14\right\}.
\]
Note that for $k \in \{9, 10\}$ and any $b$, we have 
\begin{align}
\label{eq:xyz}
\frac{\alpha_k}{b} + \min\left\{\frac{b}{2}, \frac14\right\}  \ge \sqrt{2\alpha_k}.
\end{align}
Since, when $b \le 1/2$, 
$\frac{\alpha_k}{b} + \min\left\{\frac{b}{2}, \frac14\right\} = \frac{\alpha_k}{b} + \frac{b}{2}  \ge \sqrt{2\alpha_k}$.
When $b \ge 1/2$, we also have
$\frac{\alpha_k}{b} + \min\left\{\frac{b}{2}, \frac14\right\} \ge \frac14 \ge \sqrt{2\alpha_k}$,
for $k \in \{9, 10\}$. Thus, by directly substituting $k$ into the calculation, we have
\[R_k \ge \frac{\alpha_k}{1-2\sigma_k} - \gamma_k + 2\sqrt{\sigma_k} - 2\sigma_k + \frac{\beta_k}{2(1-\sigma_k)} + \frac{\beta_k}{\frac32-\sigma_k} + \sqrt{2\alpha_k} > 0.\]
It remains in Case~1 to consider $p,q\le1/2$.  By~\eqref{eq:k9k10-E-crude},
\[ R_k\ge
\frac{\alpha_k}{a}+  \frac{\alpha_k}{b} + \min\left\{\frac a2,\frac14\right\} + \min\left\{\frac b2,\frac14\right\} -\gamma_k- \sigma_k + \sqrt{\sigma_k} + \frac{2\beta_k}{\frac32-\sigma_k}
\]
If one of $a,b$ is at least $1/2$, say $a\ge1/2$.  Then by~\eqref{eq:xyz}
\[ \frac{\alpha_k}{a}+  \frac{\alpha_k}{b} + \min\left\{\frac a2,\frac14\right\} + \min\left\{\frac b2,\frac14\right\} \ge \frac{\alpha_k}{1-\sigma_k} +
\frac14+\sqrt{2\alpha_k}.
\]
A direct substitution of $k=9,10$ gives
\[
\frac{\alpha_k}{1-\sigma_k} +  \frac14 +\sqrt{2\alpha_k} -\gamma_k- \sigma_k + \sqrt{\sigma_k} + \frac{2\beta_k}{\frac32-\sigma_k} >0.
\]
Thus, we may assume $a,b<1/2$.  Put $s=a+b$ and $t=ab$. Since $\Phi_k(a,b)= \sigma_k-\sqrt{\sigma_k}$, we have
\[ ab-\frac{a+b}{2} = t - \frac{s}{2}\le \sigma_k-\sqrt{\sigma_k} .\]
Then,
\begin{align*}
& \frac{\alpha_k}{a}+  \frac{\alpha_k}{b} + \min\left\{\frac a2,\frac14\right\} + \min\left\{\frac b2,\frac14\right\} = \frac{s\alpha_k}{t} + \frac{s}{2} \ge \frac{s\alpha_k}{\frac{s}{2} + \sigma_k-\sqrt{\sigma_k}} + \frac{s}{2} \\
= & 2\alpha_k + \frac{2\alpha_k(\sqrt{\sigma_k} - \sigma_k)}{\frac{s}{2} + \sigma_k - \sqrt{\sigma_k}} + \frac{s}{2} \ge 2\alpha_k +  \sqrt{\sigma_k} - \sigma_k + 2\sqrt{2\alpha_k(\sqrt{\sigma_k} - \sigma_k)}
\end{align*}
Again, direct substitution for $k=9,10$ gives
\[
R_k \ge 2\alpha_k +  2\sqrt{\sigma_k} - 2\sigma_k + 2\sqrt{2\alpha_k(\sqrt{\sigma_k} - \sigma_k)} -\gamma_k + \frac{2\beta_k}{\frac32-\sigma_k}>0.\]
This completes Case~1.

\smallskip
\noindent\textbf{Case 2: $\Phi_k(a,b)=ab-\dfrac{a+b}{2}>\sigma_k-\sqrt{\sigma_k}$.} In this case
\[
R_k=
\frac{\alpha_k}{a}+\frac{\alpha_k}{b}-\gamma_k
-ab+\frac{a+b}{2}
+E_k(a,b,p)+E_k(b,a,q).
\]
We first note that this case forces $a,b\le1/2$.  Indeed, if $x=1-a<1/2$, then the function
\[
y\mapsto xy-\frac{x+y}{2}
\]
is decreasing in $y$, and since $y\ge \sigma_k/x$,
\[
xy-\frac{x+y}{2} \le \sigma_k-\frac12\left(x+\frac{\sigma_k}{x}\right) \le \sigma_k-\sqrt{\sigma_k},
\]
contradicting the assumption of this case.  Hence $x\ge1/2$, i.e. $a\le1/2$.  The same argument gives $b\le1/2$.

Besides, we have a more precise estimate of $E_k$. If $p\le1/2$, then
\begin{align}
E_k(a,b,p)
&=\frac{\beta_k}{b+p}+\left(\frac12-a\right)p+\frac a2  =\frac a2+
\frac{\beta_k}{b+p}
	+\left(\frac12-a\right)(b+p)
	-\left(\frac12-a\right)b \notag\\
	&\ge
	\frac a2+2\sqrt{\beta_k\left(\frac12-a\right)}-
	\left(\frac12-a\right)b.
	\label{eq:k9k10-E-high}
\end{align}

If $p,q\le1/2$, then applying~\eqref{eq:k9k10-E-high} twice gives
\[
R_k\ge \frac{\alpha_k}{a} + \frac{a}{2} + 2\sqrt{\beta_k\left(\frac{1}{2} - a\right)} + \frac{\alpha_k}{b} + \frac{b}{2} + 2\sqrt{\beta_k\left(\frac{1}{2} - b\right)} + ab- \gamma_k.
\]
Let 
\[D_k(x) := \frac{\alpha_k}{x} + \frac{x}{2} + 2\sqrt{\beta_k\left(\frac{1}{2} - x\right)} + \frac{x - \frac{1}{k-1}}{k-1} - \frac{3k-8}{(k-1)^2},\]
then the definition of $D_k$ gives the identity
\[
\frac{\alpha_k}{a} + \frac{a}{2} + 2\sqrt{\beta_k\left(\frac{1}{2} - a\right)} + \frac{\alpha_k}{b} + \frac{b}{2} + 2\sqrt{\beta_k\left(\frac{1}{2} - b\right)} + ab- \gamma_k
=
D_k(a)+D_k(b)+(a - \frac{1}{k-1})(b-\frac{1}{k-1}).
\]
Let 
\[
L_k^-(x):=\frac{62-5k}{6}x^2+20(10-k)x^3 \qquad
L_k^+(x):=\frac{5k-41}{48}x^2+4(10-k)x^2\left(\frac12-\frac{1}{k-1}-x\right).
\]
For $k=9,10$, Sturm's theorem verifies the two one-variable inequalities
\begin{align}
D_k(x)&\ge L_k^-(\frac{1}{k-1}-x) \ge 0 &&\left(0\le x \le \frac{1}{k-1} \right), \label{eq:k9k10-D-left-unified}\\
D_k(x)&\ge L_k^+(x - \frac{1}{k-1}) \ge 0
&&\left(\frac{1}{k-1} \le x \le\frac12\right). \label{eq:k9k10-D-right-unified}
\end{align}
If $a$ and $b$ lie on the same side of $\frac{1}{k-1}$, then~\eqref{eq:k9k10-D-left-unified} and~\eqref{eq:k9k10-D-right-unified} imply $D_k(a),D_k(b)\ge0$, while $(a-\frac{1}{k-1})(b-\frac{1}{k-1})\ge0$.  Hence $R_k \ge 0$.  

In the mixed case, by symmetry, assume $a \le \frac{1}{k-1}$ and $b \ge \frac{1}{k-1}$. Let
\[ a = \frac{1}{k-1} - X, \qquad b = Y + \frac{1}{k-1},
\]
where $0\le X\le \frac{1}{k-1}$ and $0\le Y\le1/2- \frac{1}{k-1}$.  It is enough to prove
\[
L_k^-(X)+L_k^+(Y)-XY\ge0.
\]
For $k=10$ this is
\[
2X^2+\frac3{16}Y^2-XY\ge0,
\]
which is positive semidefinite.  For $k=9$, after multiplying by $12$ we need
\[
240X^3+34X^2-12XY-48Y^3+19Y^2\ge0.
\]
If $Y=0$, this is clear.  Otherwise set $\lambda=X/Y$.  Dividing by $Y^2$ gives
\[
34\lambda^2-12\lambda+19+Y(240\lambda^3-48).
\]
If $240\lambda^3-48\ge0$, this is at least $34\lambda^2-12\lambda+19>0$.  If $240\lambda^3-48<0$, then it is minimized at $Y=3/8$, and the resulting polynomial
\[
90\lambda^3+34\lambda^2-12\lambda+1
\]
is positive for all $\lambda\ge0$ by Sturm's theorem.  Thus, in this case we also have $R_k \ge 0$.

Suppose exactly one of $p,q$ is at least $1/2$.  By symmetry, assume $p\ge1/2$ and $q\le1/2$.  Using~\eqref{eq:k9k10-E-low} and~\eqref{eq:k9k10-E-high}, we obtain
\begin{align*}
	R_k
	&\ge
	\frac{\alpha_k}{a}+\frac{\alpha_k}{b}-\gamma_k
	-ab+\frac{a+b}{2}
	+ \frac{\beta_k}{2(1-\sigma_k)}+\sqrt{\sigma_k}-\sigma_k
	+\frac b2
	+2\sqrt{\beta_k\left(\frac12-b\right)}
	-\left(\frac12-b\right)a \\
	&\ge
	2\alpha_k
	+ \frac{\beta_k}{2(1-\sigma_k)}+\sqrt{\sigma_k}-\sigma_k -\gamma_k +
	\frac{\alpha_k}{b}+b+2\sqrt{\beta_k\left(\frac12-b\right)}.
\end{align*}
Note that, by Sturm's theorem, for $k \in \{9, 10\}$ and $b \in (0, 1/2]$
we have 
\[\frac{\alpha_k}{b}+b+2\sqrt{\beta_k\left(\frac12-b\right)} \ge \frac{16-k}{20}.\]
Thus, by directly substituting $k$ into the calculation, we have
\[R_k \ge 2\alpha_k
+ \frac{\beta_k}{2(1-\sigma_k)}+\sqrt{\sigma_k}-\sigma_k -\gamma_k + \frac{16-k}{20} > 0.\]

Finally, suppose $p,q\ge1/2$.  By~\eqref{eq:k9k10-E-low}, we obtain
\[R_k \ge
\frac{\alpha_k}{a}+\frac{\alpha_k}{b}-\gamma_k
-ab+\frac{a+b}{2}
+ \frac{\beta_k}{(1-\sigma_k)}+ 2\sqrt{\sigma_k}-2\sigma_k. \]
Let $z=\sqrt{ab}$, then
\[
\frac{\alpha_k}{a}+\frac{\alpha_k}{b}-ab+\frac{a+b}{2}
\ge
\frac{2\alpha_k}{z}+z-z^2.
\]
Since $0<z\le1/2$ and $k\in\{9,10\}$,
\[
\frac{2\alpha_k}{z}+z-z^2- \frac{2k-5}{(k-1)^2}
=
\frac{((k-1)z-1)^2(k-3-(k-1)z)}{(k-1)^3z}
\ge0.
\]
Again, direct substitution for $k=9,10$ gives
\[R_k \ge \frac{2k-5}{(k-1)^2} + \frac{\beta_k}{(1-\sigma_k)}+ 2\sqrt{\sigma_k}-2\sigma_k -\gamma_k > 0.\]
This completes Case~2 and proves~\eqref{eq:k9k10-transformed}.
\end{proof}

\section{Concluding remarks}\label{section:concluding}
We have proved that the palette construction of Reiher, R\"odl and Schacht is optimal for every $k\ge9$:
\[
\piv(S_k)=\frac{k^2-5k+7}{(k-1)^2}.
\]
The global part of the proof uses only the minimal-counterexample estimate in~\cref{lem:minimal-counterexample-degree}.  If a smallest counterexample existed at the conjectured density, this estimate would force $\delta(\mathcal P) > \sigma_k = \frac{k^2-10k+18}{(k-1)^2}$.
For $k\ge11$ one has $\sigma_k\ge1/4$, and the simple quadratic majorant in~\cref{lem:Phi-majorant} applies.  The values $k=9,10$ are below this $1/4$ threshold, but $\sigma_9=9/64$ and $\sigma_{10}=2/9$ are still large enough for the sharper, piecewise majorant in~\cref{lem:k10-majorant}.

The same mechanism does not continue automatically to the smaller stars.  For $4\le k\le7$, the quantity $\sigma_k$ is non-positive:
\[
\sigma_4=-\frac23,
\qquad
\sigma_5=-\frac7{16},
\qquad
\sigma_6=-\frac6{25},
\qquad
\sigma_7=-\frac1{12}.
\]
Thus the deletion argument used here gives no positive local lower bound on the products\[
e_{i,j}(a)e_{i,\ell}(a).
\]
Without such a local lower bound, the envelope estimates in~\cref{lem:Phi-density-bound,lem:k10-envelope} no longer control the potentially degenerate colors.

The borderline case $k=8$ is more subtle.  Here $\sigma_8=2/49>0$, but this value is too small for the reciprocal majorant used for $k=9,10$.  Indeed, if one formally substitutes $k=8$ into the statement of~\cref{lem:k10-majorant}, then the admissible point
\[
(x,y,s,t)=\left(\frac67,\frac15,1,\frac14\right)
\]
satisfies $xs\ge2/49$, $xy\ge2/49$ and $yt\ge2/49$, but the right-hand side of~\eqref{eq:k10-majorant} minus the left-hand side is
\[
\frac{-1193+490\sqrt2}{3430}<0.
\]
Thus the present analytic majorant cannot simply be extended to $k=8$.  This obstruction is methodological rather than structural: it does not suggest that the conjectured value is false, only that the proof must use additional information about palettes.

We believe that the natural remaining problem is still the following.

\begin{conjecture}\label{conj-1}
	\[
	\piv(S_k)=\frac{k^2-5k+7}{(k-1)^2}\qquad\text{for all }4\le k\le8.
	\]
\end{conjecture}

One possible way forward is to strengthen the passage from density to local degree.  The proof of~\cref{lem:minimal-counterexample-degree} uses only a coarse arc-count bound for the induced auxiliary digraphs.  For small $k$, one may need to exploit the extremal structure of $T_k$-free digraphs, or the interaction between the six admissibility relations, rather than estimating them separately.  Another possibility is to replace the one-coordinate envelopes used above by a genuinely multivariate majorant that keeps track of several admissibility degrees at once.  Such a refinement would be especially relevant for $k=8$, where the minimum-degree information is positive but the one-dimensional envelope already fails.

% \subsection*{Acknowledgments}
% H. Lin was supported by the Foundation of Naval University of Engineering (2025508020).
% W. Zhou was supported by the Natural Science Foundation of China  (12401457), the China Postdoctoral Science Foundation (2024M761780), the Natural Science Foundation of Shandong Province (ZR2024QA067) and Young Talent of Lifting engineering for Science and Technology in Shandong, China (SDAST2025QTA074).

\sloppy\printbibliography

\appendix
\section{Exact verification of the auxiliary inequalities}\label{app:sturm-verification}

This appendix records the exact checks used in the proofs of~\cref{lem:Phi-majorant,lem:k10-majorant}.  All sign decisions below are made over rational or real algebraic fields; decimal values are displayed only to indicate the size of the positive margins.

For $k\in\{9,10\}$ let
\[
\sigma_k=\frac{k^2-10k+18}{(k-1)^2},\qquad
\alpha_k=\frac{k-3}{2(k-1)^3},\qquad
\beta_k=\frac{2(k-3)}{(k-1)^3},\qquad
\gamma_k=\frac{3(2k-5)}{(k-1)^2}.
\]
The direct substitutions in the proof of~\cref{lem:k10-majorant} reduce to the positivity of the following six algebraic quantities:
\begin{align*}
	M_{1,k}&:=\frac{2\alpha_k+\beta_k}{1-\sigma_k}+3\sqrt{\sigma_k}-3\sigma_k-\gamma_k,\\
	M_{2,k}&:=\frac{\alpha_k}{1-2\sigma_k}-\gamma_k+2\sqrt{\sigma_k}-2\sigma_k
	+\frac{\beta_k}{2(1-\sigma_k)}+\frac{\beta_k}{\frac32-\sigma_k}+\sqrt{2\alpha_k},\\
	M_{3,k}&:=\frac{\alpha_k}{1-\sigma_k}+\frac14+\sqrt{2\alpha_k}-\gamma_k-\sigma_k+\sqrt{\sigma_k}
	+\frac{2\beta_k}{\frac32-\sigma_k},\\
	M_{4,k}&:=2\alpha_k+2\sqrt{\sigma_k}-2\sigma_k+2\sqrt{2\alpha_k(\sqrt{\sigma_k}-\sigma_k)}
	-\gamma_k+\frac{2\beta_k}{\frac32-\sigma_k},\\
	M_{5,k}&:=2\alpha_k+\frac{\beta_k}{2(1-\sigma_k)}+\sqrt{\sigma_k}-\sigma_k
	-\gamma_k+\frac{16-k}{20},\\
	M_{6,k}&:=\frac{2k-5}{(k-1)^2}+\frac{\beta_k}{1-\sigma_k}+2\sqrt{\sigma_k}-2\sigma_k-\gamma_k.
\end{align*}
Exact evaluation in the real algebraic field gives the following decimal lower bounds:
\[
\begin{array}{c|cc}
	& k=9 & k=10\\ \hline
	M_{1,k} & 0.1346590909\ldots & 0.2290283771\ldots\\
	M_{2,k} & 0.0066580923\ldots & 0.0768170074\ldots\\
	M_{3,k} & 0.0245541159\ldots & 0.0778494163\ldots\\
	M_{4,k} & 0.0103921950\ldots & 0.0803006845\ldots\\
	M_{5,k} & 0.0003551136\ldots & 0.0155746168\ldots\\
	M_{6,k} & 0.0897727272\ldots & 0.1526855847\ldots
\end{array}
\]
In particular all six quantities are positive for both $k=9$ and $k=10$.

\smallskip

It remains to certify a few one-variable inequalities.  Besides the polynomial inequalities
\begin{align*}
&3x^7+12x^6-19x^5-104x^4+36x^3+200x^2-168x+32 <0  \quad(\sqrt2<x<\sqrt5), \\
&90\lambda^3+34\lambda^2-12\lambda+1>0  \quad(\lambda\ge0),
\end{align*}
the only checks involving radicals are the two inequalities~\eqref{eq:k9k10-D-left-unified}--\eqref{eq:k9k10-D-right-unified} and
\[
\frac{\alpha_k}{b}+b+2\sqrt{\beta_k\left(\frac12-b\right)}\ge\frac{16-k}{20}
\qquad(0<b\le1/2).
\]
For an expression $A(x)+B(x)\sqrt{C(x)}$, all possible zeros lie among the zeros of $A(x)^2-B(x)^2C(x)$, together with the zeros of $C(x)$.  Thus Sturm root isolation, followed by an exact sign check on each resulting interval and at the endpoints, certifies the inequalities.  The following SageMath script implements precisely this procedure.

\begin{nolinenumbers}
\begin{SageCode}
	R.<X> = PolynomialRing(QQ)
	AA = AlgebraicRealField()
	h = QQ(1)/2
	
	def roots_in(p, a, b):
	    a = AA(a)
	    b = AA(b)
	    roots = p.roots(AA, multiplicities=False)
	    return sorted([
	        AA(r) for r in roots
	        if a < AA(r) < b
	    ])
	
	def uniq(v):
	    v = sorted(v)
	    out = []
	    for z in v:
	        if not out or z != out[-1]:
	            out.append(z)
	    return out
	
	def val(A, B, C, t):
	    return AA(A(t)) + AA(B(t))*AA(C(t)).sqrt()
	
	def verify_sqrt(A, B, C, a, b):
	    cuts = uniq(
	        [AA(a), AA(b)]
	        + roots_in(A^2 - B^2*C, a, b)
	        + roots_in(C, a, b)
	    )
	    for t in cuts:
	        assert val(A, B, C, t) >= 0
	    for u, v in zip(cuts, cuts[1:]):
	        assert val(A, B, C, (u + v)/2) > 0
	
	def const(k):
	    k = QQ(k)
	    return (
	        (k^2 - 10*k + 18)/(k - 1)^2,
	        (k - 3)/(2*(k - 1)^3),
	        2*(k - 3)/(k - 1)^3,
	        3*(2*k - 5)/(k - 1)^2,
	    )
	
	def Lm(k, z):
	    k = QQ(k)
	    return (62 - 5*k)/6*z^2 + 20*(10 - k)*z^3
	
	def Lp(k, z):
	    k = QQ(k)
	    r = QQ(1)/(k - 1)
	    return (5*k - 41)/48*z^2 + 4*(10 - k)*z^2*(h - r - z)
	
	def check_D(k, side):
	    k = QQ(k)
	    sig, alp, bet, gam = const(k)
	    r = QQ(1)/(k - 1)
	    reg = X/2 + (X - r)/(k - 1) - (3*k - 8)/(k - 1)^2
	    if side == 'left':
	        L, a, b = Lm(k, r - X), QQ(0), r
	    else:
	        L, a, b = Lp(k, X - r), r, h
	    verify_sqrt(alp + X*(reg - L), 2*X, bet*(h - X), a, b)
	
	P = 3*X^7 + 12*X^6 - 19*X^5 - 104*X^4 \
	    + 36*X^3 + 200*X^2 - 168*X + 32
	assert not roots_in(P, AA(2).sqrt(), AA(5).sqrt())
	assert P(AA(2).sqrt()) < 0
	assert P(AA(5).sqrt()) < 0
	
	Q = 90*X^3 + 34*X^2 - 12*X + 1
	assert Q(0) > 0
	assert all(AA(r) < 0 for r in Q.roots(AA, multiplicities=False))
	
	for k in [9, 10]:
	    sig, alp, bet, gam = const(k)
	    check_D(k, 'left')
	    check_D(k, 'right')
	    verify_sqrt(
	        alp + X*(X - QQ(16 - k)/20),
	        2*X,
	        bet*(h - X),
	        QQ(0),
	        h,
	    )
	
	print('All checks passed.')
\end{SageCode}
\end{nolinenumbers}

Finally, the lower bounds $L_k^-$ and $L_k^+$ used in the mixed case are nonnegative on their ranges because
\[
L_9^-(X)=\frac{17}{6}X^2+20X^3,
\qquad
L_{10}^-(X)=2X^2,
\]
and
\[
L_9^+(Y)=Y^2\left(\frac{19}{12}-4Y\right)\ge0\quad(0\le Y\le3/8),
\qquad
L_{10}^+(Y)=\frac{3}{16}Y^2\ge0.
\]

\end{document}